\documentclass[10pt]{article}       
\usepackage{geometry}               
\usepackage{graphicx}
\usepackage{caption}
\usepackage{subcaption}
\usepackage{amsmath} 
\usepackage{amssymb}  
\usepackage{amsthm}  
\usepackage{bm}
\usepackage{lipsum}
\usepackage[linesnumbered, ruled]{algorithm2e}
\usepackage{color}

\newtheorem{proposition}{Proposition}

\newtheorem{theorem}{Theorem}

\title{\LARGE \bf
Dynamic Contracts with Partial Observations:\\ Application to Indirect Load Control}

\author{Insoon Yang 
\thanks{I. Yang is with the Department of Electrical Engineering and Computer Sciences, University of California, Berkeley, CA 94720, USA
        {\tt\small iyang@eecs.berkeley.edu}}%
\and Duncan S. Callaway
\thanks{D. S. Callaway is with the Energy and Resources Group, University of California, Berkeley, CA 94720, USA
        {\tt\small dcal@berkeley.edu}}
\and Claire J. Tomlin
\thanks{C. J. Tomlin is with the Department of Electrical Engineering and Computer Sciences, University of California, Berkeley, CA 94720, USA 
        {\tt\small tomlin@eecs.berkeley.edu}}%
}

\date{}

\begin{document}

\maketitle
\pagestyle{plain}

\begin{abstract}
This paper proposes a method to design an optimal dynamic contract between a principal and an agent, who has the authority to control both the principal's revenue and an engineered system. 
The key characteristic of our problem setting is that the principal has very limited information: the principal has no capability to monitor the agent's control or the state of the engineered system. The agent has perfect observations.
With this asymmetry of information, we show that the principal can induce the agent to control both the revenue and the system processes in a way that maximizes the principal's utility, if the principal offers appropriate real-time and end-time compensation. 
We reformulate the dynamic contract design problem as a stochastic optimal control of both the engineered system and the agent's future expected payoff, which can be numerically solved using an associated Hamilton-Jacobi-Bellman equation.
The performance and usefulness of the proposed contract are demonstrated with an indirect load control problem.
\end{abstract}

\section{Introduction}
Designing contracts and incentives is an important problem in economics and engineering.
Economics studies on the principal-agent problem analyze how a contract between a principal (e.g., a company) and an agent (e.g., a worker) can be designed in an uncertain environment \cite{Laffont2002}.
Contract design problems have also been investigated to incentivize agents in engineered infrastructure, such as electric power systems, to improve its operation and regulation \cite{Gedra1993, Bitar2012}.

Dynamic contract studies in economics and finance primarily concern the principal's (stochastic) revenue stream as a dynamic system of interest and study how to design compensation schemes to motivate the agent to control the revenue process in a way that maximizes the principal's profit \cite{Holmstrom1987, DeMarzo2006}.
In engineering, on the other hand, the dynamic contracts can also be used to manage the agent's control of an engineered system. 
For example, an electricity utility company may want to make a (dynamic) contract with a building manager who has the authority to control the indoor temperature of a building. 
The utility offers a dynamic compensation scheme for the manager to control the building temperature dynamics
in a way that maximizes the utility's profit (or minimizes its costs) while maintaining the indoor temperature level within a desirable range. 
This idea of \emph{indirect load control} can be realized by a dynamic contract design method that is capable of managing engineered systems.

In this study, we propose a method to design a dynamic contract  when the agent controls both the principal's revenue and an (engineered) system that affect the principal's utility. 
The key advantage of the proposed method is that it is effective for situations in which the principal's monitoring capabilities are very limited and thus the principal has partial observations. We will investigate the specific case in which
the principal is not able to monitor the agent's control or the state of the engineered system and can only observe the output of her noisy revenue process.
Therefore, the proposed method is useful for systems in which the agent's privacy is critical or monitoring the system is costly.
The contract that we study is \emph{dynamic} in the sense that, based on the observations of her revenue process, the principal dynamically chooses the compensation for the agent to generate an \emph{incentive compatible} control, which maximizes the agent's utility. 
The contract is \emph{optimal} in the sense that the combination of the compensation scheme and the incentive compatible control strategy maximizes the principal's utility.

The most relevant work to ours is that of Sannikov \cite{Sannikov2008}, in which the contract takes into account only the principal's revenue. This study suggests that the agent's expected future payoff can be used as a performance index and that an optimal compensation scheme and an incentive compatible control can be chosen as feedback maps from the index in an infinite horizon setting. 
In this work, we consider a dynamic contract in a finite time horizon.
We suggest that the agent's future expected payoff can help in designing an optimal contract in our setting in which the principal has partial observations,  if the end-time compensation is properly chosen. However, having only the agent's future expected payoff is insufficient for designing an optimal contract when an engineered system also affects the principal's utility. We thus reformulate the dynamic contract design problem as a stochastic optimal control of both the agent's future expected payoff and the engineered system dynamics.
Applying the dynamic programming principle, the value function of the reformulated problem can be computed as a \emph{viscosity solution} of a Hamilton-Jacobi-Bellman (HJB) equation.
As a result, an optimal compensation scheme and incentive compatible control strategy can be written as feedback maps from both the agent's expected future payoff and the engineered system state.

A model in which a noisy nonlinear system state is directly observed by the principal is studied using forward-backward stochastic differential equations (FBSDEs) \cite{Cvitanic2009}. This stochastic maximum principle-based approach is extensively investigated in a recent monograph \cite{Cvitanic2012}. \cite{Williams2009} also characterizes optimal contracts using FBSDEs when the agent's private system is controlled with a separate variable. However, it does not provide a detailed solution except for special cases in which analytic solutions are available. 
A completely different model is studied in \cite{Biais2010}, in which the agent's effort is modeled by a Poisson process.

The rest of the paper is organized as follows. 
In Section \ref{setting}, we present the problem setting in which a principal makes a contract with an agent who has the authority to control the principal's revenue and an (engineered) system, and we discuss the availability of information to the principal.
In Section \ref{agentIC}, we characterize the incentive compatibility of  the agent's control in a finite horizon setting.
In Section \ref{principalOC}, we formulate the principal's optimal contract design problem as a stochastic optimal control problem and show the optimality of the resulting contract.
Lastly, the performance of the proposed dynamic contract is demonstrated with an example of indirect load control in Section \ref{example}.


\section{The Setting} \label{setting}

Consider the principal's revenue process, $x := \{ x_t\}_{0\leq t \leq T}$, $x_t \in \mathbb{R}$, and the (engineered) system process, $y := \{ y_t \}_{0\leq t \leq T}$, $y_t \in \mathbb{R}$, that affects the principal's utility: 
\begin{equation} \label{dynamics}
\begin{split} 
dx_t &= u_t dt + \sigma dW_t\\
dy_t &= f(y_t, u_t) dt,
\end{split}
\end{equation}
given initial values with $\sigma \in \mathbb{R}$.
Here, $\{W_t\}_{t \geq 0}$ is an one-dimensional standard Brownian motion on a probability space $(\Omega, \mathcal{F}, \mathbf{P})$. Let $\{\mathcal{F}_t^W\}_{t \geq 0}$ be the filtration generated by the Brownian motion.
For convenience, we consider a one-dimensional engineered system, but the method proposed in this study can be applied to any system with multiple dimensions as well.
Note that the control process $u:= \{u_t\}_{0\leq t\leq T}$ affects both the revenue and system processes. 
The principal wants to hire an agent to manage both processes. That is, the agent has the authority to determine $u$, while the principal does not.
It would be ideal for the principal if the agent's utility and principal's utility are the same: the (rational) agent then behaves in a way that maximizes the principal's utility. 
In reality, however, the agent's utility is not aligned with that of the principal. Therefore, the principal has to incentivize the agent to be cooperative by providing appropriate compensation.

\begin{figure}[tb] 
\begin{center}
\includegraphics[width =3.3in]{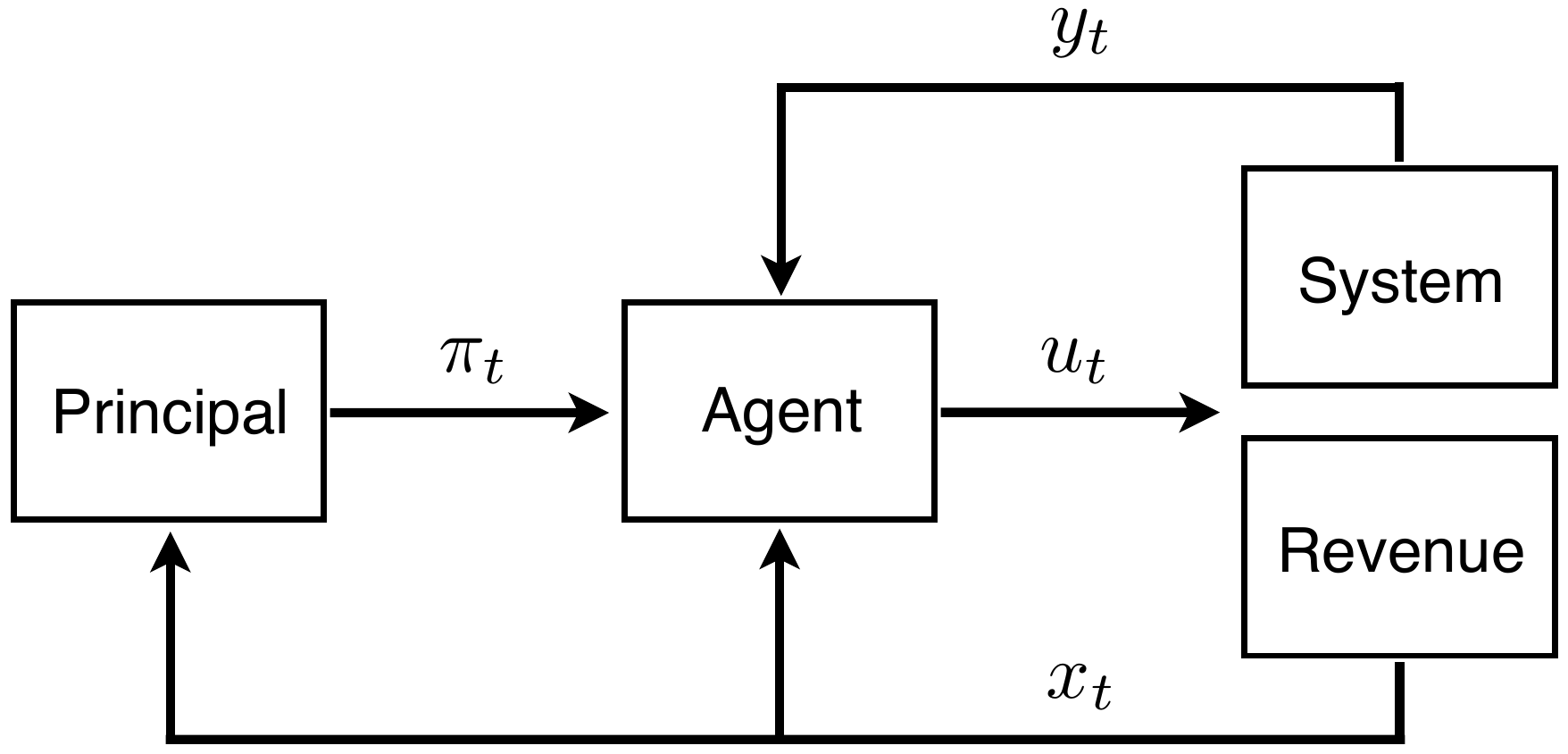}
\caption{Information flow at time $t$ under dynamic contract}
 \label{fig:diagram}
 \end{center}
\end{figure}  

A contract $(\pi, C, u^*)$ between the principal and the agent specifies the real-time compensation $\pi := \{\pi_t\}_{0\leq t \leq T}$, $\pi_t \in \mathbb{R}$, the end-time compensation $C \in \mathbb{R}$ and the recommended control strategy $u^* := \{ u_t^* \}_{0\leq t \leq T}$, $u_t^* \in \mathbb{R}$. 
The principal offers the contract at time $0$. The agent accepts the contract if the agent's expected total payoff is greater than some threshold, called the \emph{participation payoff}, assuming that he follows the recommended control strategy. This condition is often called \emph{individual rationality}.

If the principal can monitor the agent's control process, the principal can enforce any $u^*$ by imposing a significant penalty (written down in the contract) on the agent when his choice deviates from $u^*$.
However, monitoring the agent's effort is not possible (or is costly) in many practical applications. 
In addition, we consider a situation in which the principal is not able to (or is not willing to) monitor the system process $y$.
The principal can only observe her revenue process $x$, as shown in Fig. \ref{fig:diagram}. 
Therefore, the agent is free to deviate from the recommended control (or effort) level without informing the principal of his deviation. 
For example, the agent can shirk and then attribute the resulting low revenue to noise.
On the other hand, the agent has perfect observations of the revenue, system and control processes.
Fig. \ref{fig:diagram} depicts the information flow in our problem setting.

Due to her limited monitoring capability and thus partial observations, 
the principal cannot enforce an arbitrary recommended control upon the agent: the principal can only enforce an \emph{incentive compatible} control, which maximizes the agent's expected utility given the compensation scheme in the contract.
Therefore, the principal's best strategy in this setting of asymmetric information is to choose a compensation scheme such that its incentive compatible control of the agent and the compensation scheme  maximize the principal's expected utility.

This study uses the following expected payoffs for the principal and agent:
\begin{equation}\nonumber
\begin{split}
\mbox{(principal)} \quad &\mathbb{E} \left [ \int_0^T \left (dx_t +  r^P (y_t, \pi_t) dt \right ) + q(y_T) - C  \right ],\\
\mbox{(agent)} \quad &\mathbb{E} \left [ \int_0^T (r^A (\pi_t) - h(u_t)) dt + g(C)  \right ].
\end{split}
\end{equation}
Here, $r^P(y_t, \pi_t)$ and $q(y_T)$ denote the real-time and terminal rewards that the principal obtains. 
In addition, $h(u_t)$ is the cost for control $u_t$, and $r^A(\pi_t)$ and $g(C)$ denote the utilities coming from the real-time and end-time compensation, respectively. Note that the agent is indifferent to the principal's revenue and the engineered system performance.  
This utility model of the agent is useful for the situations in which the principal owns the engineered system (infrastructure). 
As we will see in the indirect load control problem in Section \ref{example}, using this model is also reasonable when it is difficult to quantify the agent's payoff as a function of the system state.
We assume that the principal is risk neutral while the agent is risk averse.
Let $r^A$ be concave and $h$ be differentiable, convex and increasing. We also assume that $g$ is a continuous invertible function, which is concave and strictly increasing.
We introduce the sets of feasible control, real-time compensation and end-time compensation schemes as
\begin{equation} \nonumber
\begin{split}
\mathbb{U} &:= \{ u : [0,T] \to \mathcal{U} \: | \: u \mbox{ progressively measurable with respect to $\mathcal{F}_t^W$} \},\\
\Pi &:= \{ \pi : [0,T] \to \mathcal{P} \: | \: \pi \mbox{ progressively measurable with respect to $\mathcal{F}_t^x$}  \},\\
\mathbb{C} &:= \{ C \in \mathbb{R} \: | \: C \mbox{ is $\mathcal{F}_T^x$-measurable} \},
\end{split}
\end{equation}
respectively,
where $\mathcal{U}$ and $\mathcal{P}$ are compact sets in $\mathbb{R}$ and $\{\mathcal{F}_t^x\}_{0\leq t \leq T}$ is the filtration generated by $x$.
We assume that $f :\mathbb{R} \times \mathcal{U} \to \mathbb{R}$ is continuous and that there exists $L >0$ such that $|f(y_1, a) - f(y_2, a)| \leq L|y_1 - y_2|$ for all $y_1, y_2 \in \mathbb{R}$ and $a \in \mathcal{U}$.

\section{The Agent's Incentive Compatible Control}
\label{agentIC}

To design the real-time and end-time compensation schemes that are optimal for the principal, we first characterize a condition in which the agent's control is incentive compatible. 
Given the real-time and end-time compensation, $(\pi, C)$, the agent's \emph{incentive compatible} control is the solution of
\begin{equation}\nonumber
\max_{u \in \mathbb{U}} \; \mathbb{E}^{u} \left [ \int_0^T (r^A(\pi_t) - h(u_t) ) dt + g(C) \right ],
\end{equation}
i.e., any control that is incentive compatible with $(\pi,C)$ maximizes the agent's expected total utility given $(\pi,C)$.
Here, the expectation is taken under the probability measure $\mathbf{P}^{u}$ induced by the agent's control $u$.
We introduce a new variable that denotes the agent's expected future payoff, i.e., 
\begin{equation}\nonumber
w_t := \mathbb{E}^{u} \left [ \left. \int_t^T (r^A(\pi_s) - h(u_s))  dt + g(C) \right | \mathcal{F}_t^W \right ].
\end{equation}
Note that the new variable also depends on the end-time compensation, while the one in Sannikov \cite{Sannikov2008} does not because there is no end-time compensation in the infinite horizon setting of Sannikov.
The dynamics of the agent's expected future payoff can be written as \eqref{cont} in the following proposition.

\begin{proposition} \label{prop1}
Given $(\pi, C)$,
there exists an $\{\mathcal{F}_t^W\}_{0\leq t \leq T}$-adapted process $\xi := \{ \xi_t \}_{0\leq t \leq T}$ (depending on $(\pi, C, u)$) such that
\begin{equation}\label{cont}
w_t = w_0 - \int_0^t( r^A(\pi_s) - h(u_s)) ds + \int_0^t \xi_s (dx_s - u_s ds)
\end{equation}
for all $t \in [0,T]$.
\end{proposition}
Similar to those in Sannikov \cite{Sannikov2008} and Ekeland \cite{Ekeland2013}, the proposition can be proved using the martingale representation theorem \cite{Karatzas1991}.
Proposition \ref{prop1} allows us to write the dynamics of the agent's expected future payoff as the following stochastic differential equation:
\begin{equation} \label{sde}
\begin{split}
dw_t &= -(r^A(\pi_t) - h(u_t) + \xi_t u_t) dt + \xi_t dx_t\\
w_T &= g(C).
\end{split}
\end{equation}
Even if the agent changes his control level and the system process is changed accordingly, the principal's compensation scheme stays unchanged because the principal is not able to observe the agent's control strategy or the (engineered) system process.
Therefore, the agent has an incentive to choose his control level that maximizes $-h(u_t) + \xi_t u_t$, given $(\pi, C)$.

\begin{proposition} \label{prop2}
Let $\xi$ be the $\{\mathcal{F}_t^W\}_{0\leq t \leq T}$-adapted process defined by \eqref{cont} given $(\pi, C, u)$. The control process $u$ is incentive compatible with $(\pi, C)$ if and only if 
\begin{equation}\label{condition}
-h(u_t ) + \xi_t u_t = \max_{a \in \mathcal{U}} \left \{ -h (a) + \xi_t a \right \}, \quad t \in [0,T],
\end{equation}
almost everywhere.
\end{proposition}
This proposition can be shown using the argument in \cite{Sannikov2008} and \cite{Ekeland2013} even in our finite time horizon setting.
Note that $\xi_t$ represents the sensitivity of the agent's expected future payoff to the principal's noisy revenue.
Because the agent is risk averse while the principal is risk neutral, they select the minimum level of $\xi_t$ such that it makes the recommended control  incentive compatible. We denote this level by $\theta(u_t)$, which can be expressed as
\begin{equation}\nonumber
\theta(u_t) := \min_{0 \leq z < \infty} \: \left \{ z \: | \: u_t \in \arg \max_{a \in \mathcal{U}} \{ -h(a) + z a\}  \right \},
\end{equation}
as claimed in \cite{Sannikov2008}.
Note that the sensitivity $\xi_t$ also depends on $(\pi, C)$: the risk-neutral principal offers $(\pi,C)$ such that $\xi_t[\pi,C,u] = \theta(u_t)$ for the risk-averse agent.
If $h$ is strictly convex and the principal chooses a contract $(\pi, C, u)$ such that 
$\xi_t[\pi, C, u] = h'(u_t)$, $t \in [0,T]$,
a.e.,
then the agent's incentive compatible control is $u$, and hence, $\theta(u_t) = \xi_t[\pi,C, u] = h'(u_t)$.
To design such a contract, the principal can use the agent's expected future payoff as a performance index for the agent.
However, the principal needs another performance index for the engineered system because it also affects her utility.
Taking into account both performance indices, the contract design problem can be formulated as the stochastic optimal control of the agent's expected future payoff process \eqref{sde} and the (engineered) system, as discussed in the following section.




\section{The Principal's Problem: Optimal Dynamic Contract} \label{principalOC}

We now turn our attention to the principal's problem: she wants to choose real-time and end-time compensation such that both the compensation scheme and the corresponding agent's incentive compatible control $(i)$ maximize the principal's utility and $(ii)$ make the agent's total payoff exceed the participation payoff, $b \in \mathbb{R}$. This problem can be formulated as the following optimization:
\begin{subequations} \label{principal}
\begin{align}
 \max_{\substack{u \in \mathbb{U}, \pi \in \Pi,\\ C \in \mathbb{C}}} \quad &\mathbb{E} \left [\int_0^T u_t + r^P(y_t,  \pi_t)  dt + q(y_T) - C \right ]\\
\mbox{subject to} \quad & dx_t = u_t dt + \sigma dW_t\\
& dw_t = -(r^A(\pi_t) - h(u_t) + \xi_t u_t) dt + \xi_t dx_t  , \quad t\in [0,T]\label{continuation}\\ 
& w_T = g(C) \label{terminal}\\ 
& \xi_t = \theta (u_t), \quad t\in [0,T], \quad \mbox{a.e.} \label{ic2}\\ 
& w_0 \geq b \label{ir}\\ 
&d y_t = f(y_t, u_t) dt, \quad t\in [0,T] \label{engineer} \\ 
&y_0 = y^0, \label{init}
\end{align}
\end{subequations}
where \eqref{continuation} and \eqref{terminal} specify the dynamics of the agent's expected future payoff, \eqref{ic2} is the condition for the agent's incentive compatible control, and \eqref{ir} imposes the individual rationality condition that the agent's expected total payoff should exceed the participation payoff.
Furthermore, \eqref{engineer} and \eqref{init} should be included because the principal's utility is affected by the state, $y_t$, of the engineered system.

If $w_0 > b$, then the principal has an incentive to decrease the end-time compensation $C$ because $g$ is increasing. Therefore, the constraint \eqref{ir} can be substituted with $w_0 = b$. Recalling that $g$ is an invertible function, we reformulate the problem as 
\begin{subequations}\label{principal2}
\begin{align}
\max_{u \in \mathbb{U}, \pi \in \Pi'} \quad &\mathbb{E} \left [\int_0^T u_t + r^P(y_t, \pi_t)  dt + q(y_T) - g^{-1} (w_T)\right ] \nonumber\\
\mbox{subject to} \quad & dw_t = -(r^A(\pi_t) - h(u_t) ) dt + \theta (u_t) \sigma dW_t, \quad t \in [0,T] \label{continuation2}\\
& w_0 = b \label{init2}\\
&d y_t = f(y_t, u_t) dt, \quad t \in [0,T]\\
&y_0 = y^0,
\end{align}
\end{subequations}
where $\Pi' := \{ \pi : [0,T] \to \mathcal{P} \: | \: \pi \mbox{ progressively measurable}$ $\mbox{with respect to $\mathcal{F}_t^W$} \} \supseteq \Pi$.
We will show that its solution also solves \eqref{principal} in Theorem \ref{thm2}.
Note that this problem has the form of a standard stochastic optimal control problem.
We can solve this problem using dynamic programming. We define the value function of the optimal control as
\begin{equation}\nonumber
\phi (\bm{w}, \bm{y}, t) := \max_{u \in \mathbb{U}, \pi \in \Pi'} \: J_{\bm{w}, \bm{y},t}^P [u, \pi],
\end{equation}
where 
\begin{equation}\nonumber
\begin{split}
&J_{\bm{w}, \bm{y},t}^P [ u, \pi] := \mathbb{E}_{\bm{w}, \bm{y},t}  \left [ \int_t^T u_s + r^P(y_s, \pi_s)  ds + q(y_T) - g^{-1} (w_T) \right ].
\end{split}
\end{equation}
Here, $\mathbb{E}_{\bm{w}, \bm{y},t} [B]$ denotes the expected value of $B$ conditioned on $(w_t, y_t) = (\bm{w}, \bm{y})$.
The value function corresponds to the viscosity solution of the following Hamilton-Jacobi-Bellman equation \cite{Crandall1983}:
\begin{equation} \nonumber
\begin{split}
&\frac{\partial \phi}{\partial t} + \max_{(p,a) \in \mathcal{P} \times \mathcal{U}} \left \{-(r^A(p) - h(a)) D_{\bm{w}} \phi + f(\bm{y}, a) D_{\bm{y}} \phi  \right.\\
&\left. \qquad \qquad \qquad \quad \: + a + r^P(\bm{y},p) + \frac{1}{2} (\theta (a) \sigma)^2 D_{\bm{w}}^2 \phi  \right \} = 0,
\end{split}
\end{equation}
with the terminal condition $\phi(\bm{w},\bm{y},T) = -g^{-1}(\bm{w}) + q(\bm{y})$.
By solving the HJB equation, we can determine an optimal compensation scheme and an optimal recommended control strategy. Let $(w^*_0, y^*_0) = (b, y^0)$. Given an optimal trajectory $(w_s^*, y_s^*)$ for $s \in [0,t]$, an optimal real-time compensation and a recommended (incentive compatible) control can be obtained as
\begin{equation}\nonumber
\begin{split}
(\pi^*_t, u_t^*) &:= \arg \max_{(p,a) \in \mathcal{P} \times \mathcal{U}} \left \{  
- (r^A(p) - h(a)) D_{\bm{w}} \phi (w_t^*, y_t^*)  +  f(y_t^*, a) D_{\bm{y}} \phi (w_t^*, y_t^*)  + a + r^P(y_t^*,p) \right.\\
&\left. \qquad \qquad \qquad \quad \;\:\:+ \frac{1}{2} (\theta (a) \sigma)^2  D_{\bm{w}}^2 \phi(w_t^*, y_t^*)
\right \}
\end{split}
\end{equation} 
and an optimal end-time compensation is computed as 
\begin{equation}\nonumber
C^* = g^{-1}(w^*_T).
\end{equation}
A more detailed discussion regarding how to synthesize an optimal control from a viscosity solution of an associated HJB equation can be found in \cite{Bardi1997} even when the solution $\phi$ is not differentiable.
We notice that the optimal compensation and control are Markovian with respect to $(w_t^*, y_t^*)$.
Therefore, $\pi^*$ must be in $\Pi$, which is the original feasible set.
Recall that the principal cannot observe the system process $y$; thus, it does not seem feasible for the principal to make this contract. If the agent chooses the recommendation $u^*$ as his control, however, the principal can infer the state of the system from its dynamics, i.e.,
\begin{equation}\nonumber
\begin{split}
dy_t^* &= f(y_t^*, u_t^*) dt\\ 
y_0^* &= y^0. 
\end{split}
\end{equation}
Because the differential equation has a unique solution, her inference concerning the state of the system must be correct. 
Therefore, the principal has the correct system process information if the recommended control $u^*$ is incentive compatible with $(\pi^*, C^*)$. 
In the following theorem, we prove that the recommended control $u^*$ is indeed incentive compatible with $(\pi^*, C^*)$.

\begin{theorem} \label{thm1}
Let $(u^*, \pi^*)$ be a solution of \eqref{principal2} and $C^* := g^{-1} (w_T^*)$, where $w^*$ is the process driven by
\begin{equation}\label{sde2}
\begin{split}
dw_t^* &= -(r^A(\pi_t^*) - h(u_t^*)) dt + \theta (u_t^*) (dx_t -  u_t^* dt)\\
w_0^* &= b.
\end{split}
\end{equation}
Then, $u^*$ is the agent's control that is incentive compatible with $(\pi^*, C^*)$.
\end{theorem}

\begin{proof}
Given $(\pi^*, C^*, u^*)$, the agent's expected future payoff is driven by
\begin{equation} \nonumber
\begin{split}
dw_t &= -(r^A( \pi_t^*) - h(u_t^*) + \xi_t u_t^*) dt + \xi_t dx_t\\
w_T &= g(C^*),
\end{split}
\end{equation}
where $\xi_t = \xi_t[\pi^*, C^*, u^*]$.
Due to Proposition \ref{prop2},
$u^*$ is incentive compatible with $(\pi^*, C^*)$ if and only if $\theta (u^*_t) = \xi_t$ for $t \in [0,T]$, almost everywhere. 

Now, we consider the process governed by \eqref{sde2}. Because the end-time compensation is chosen as $C^* = g^{-1}(w_T^*)$, $w_T^* = g(C^*)$, which implies that
\begin{equation}\nonumber
w_T^* = w_T.
\end{equation}
Define $\psi_t := w_t - w_t^*$, which then satisfies
\begin{equation}\nonumber
\begin{split}
d\psi_t &= (\xi_t - \theta (u_t^*)) \sigma dW_t^{u^*}\\
\psi_T &= 0,
\end{split}
\end{equation}
where 
\begin{equation}\nonumber
W_t^{u^*} := \frac{1}{\sigma} \left ( x_t - \int_0^t u_s^* ds \right )
\end{equation}
is a Brownian motion under the probability measure $\mathbf{P}^{u^*}$. 
Using the It\^{o} isometry, we have
\begin{equation}\nonumber
0 = \mathbb{E}^{u^*} \left [\psi_T^2 \right ] = \mathbb{E}^{u^*} \left [ \int_0^T  \{(\xi_t - \theta (u_t^*)) \sigma\}^2 dt \right ],
\end{equation}
which implies that 
\begin{equation} \nonumber
\xi_t = \theta (u_t^*), \quad t \in [0,T],
\end{equation}
almost everywhere.
Therefore, $u^*$ is incentive compatible with $(\pi^*,C^*)$ as desired.
\end{proof}

%
%
We are now ready to show that $(\pi^*, C^*, u^*)$ obtained by solving \eqref{principal2} is an optimal contract.

\begin{theorem} \label{thm2}
Let $(u^*, \pi^*)$ be a solution of \eqref{principal2} and $C^* := g^{-1} (w_T^*)$, where $w^*$ is the process driven by \eqref{sde2}.
Then, $(\pi^*, C^*, u^*)$ is an optimal contract, i.e., it solves \eqref{principal}.
\end{theorem}

\begin{proof}
We notice that $(\pi^*, C^*, u^*)$ can satisfy all the constraints of \eqref{principal} because $\xi_t = \theta (u_t^*)$ almost everywhere (due to Theorem \ref{thm1}) and $w_T^* = g(C^*)$. (Also note that $\pi^* \in \Pi$ and $C^* \in \mathbb{C}$ because $(\pi^*_t, u^*_t)$ is Markovian with respect to $(w_t^*, y_t^*)$.)
Suppose that $(\pi^*, C^*, u^*)$ is not a solution of \eqref{principal}, and
choose a solution, $(\hat{\pi}, \hat{C}, \hat{u})$, of \eqref{principal}.

We first claim that $(\hat{\pi}, \hat{C}, \hat{u})$ satisfies all the constraints of \eqref{principal2} with the process $\hat{w}$ defined by
\begin{equation}\nonumber
\begin{split}
d\hat{w}_t &= -(r^A(\hat{\pi}_t) - h(\hat{u}_t)) dt + \xi_t (d\hat{x}_t - \hat{u}_t dt)\\
\hat{w}_T &= g(\hat{C}),
\end{split}
\end{equation}
where $d\hat{x}_t =  \hat{u}_t dt + \sigma dW_t$. 
Because $\hat{w}$ solves \eqref{continuation2} with $\hat{\pi}_t$, $\hat{u}_t$ and $\xi_t = \theta (\hat{u}_t)$ a.e., it suffices to show that $\hat{w}_0 = b$. Assume $\hat{w}_0 > b$. Then, there exists $C' < \hat{C}$ such that 
\begin{equation}\nonumber
g(C') = g(\hat{C}) - (\hat{w}_0 - b)
\end{equation}
because $g$ is continuous and increasing. Then, $(\hat{\pi}, C', \hat{u})$ satisfies all the constraints of \eqref{principal} and
is strictly better than $(\hat{\pi}, \hat{C}, \hat{u})$.
This is a contradiction and $\hat{w}_0$ must be equal to the participation payoff $b$. Therefore, $(\hat{\pi}, \hat{C}, \hat{u})$ with $\hat{w}$ can satisfy all the constraints of \eqref{principal2}. 

Because $(\hat{\pi}, \hat{C}, \hat{u})$ solves \eqref{principal} while $(\pi^*, C^*, u^*)$ does not, we have
\begin{equation} \label{ineqG}
\begin{split}
&\mathbb{E} \left [\int_0^T \hat{u}_t + r^P(\hat{y}_t, \hat{\pi}_t) dt +  q(\hat{y}_T) - g^{-1} (\hat{w}_T) \right ]\\
&> \mathbb{E} \left [\int_0^T u_t^* + r^P(y_t^*, \pi_t^*) dt +  q(y_T^*) - g^{-1} (w_T^*) \right ],
\end{split}
\end{equation}
where $d\hat{y}_t = f(\hat{y}_t, \hat{u}_t)dt$ and $dy^*_t = f(y_t^*, u_t^*) dt$ and $\hat{y}_0 = y_0^* = y^0$.
Inequality \eqref{ineqG} contradicts the fact that $(\pi^*, C^*, u^*)$ is a solution to \eqref{principal2}.
Therefore, $(\pi^*, C^*, u^*)$ must be a solution of \eqref{principal} and hence an optimal contract.
\end{proof}

In the contract, the compensation scheme and recommended control are written as feedback maps of $(w_t^*, y_t^*)$, which might be different from the actual $(w_t, y_t)$. However, by offering both the compensations and the recommended control that is incentive compatible, the principal induces the agent to have no incentive to deviate from $u^*$ and $(w_t^*, y_t^*)$. Therefore, the principal can maximize her utility with this contract, as expected.

\section{Application to Indirect Load Control} \label{example}

We consider a scenario in which an electricity utility company makes a contract that is renewed daily with an aggregation of its customers to indirectly control the customers' air conditioners. 
Let $u_t^i$ be the power consumption (in kW) by  customer $i$'s air conditioner at time $t$ for $i = 1, \cdots, N$.
We impose an ideal assumption that the customers in the aggregation agree to equally control the power consumption, i.e., $u^i \equiv u$ for $i=1, \cdots, N$. Then, the total power consumption by the air conditioners at time $t$ is given by $Nu_t$.
The authority of control $u$ is given to the customers: the utility company has no capability to directly control $u$. 
Let $\lambda_t$ be the balancing price, which is chosen as the locational marginal price (LMP) of unit power in kW at time $t$ in the real-time market. We also let $\zeta$ be the energy price in the customers' tariff.
Then, the utility company's revenue process due to the air conditioners in the real-time market is driven by the following SDE:
\begin{equation}\nonumber
dx_t = (\zeta -\lambda_t) N u_t dt + \sigma dW_t,
\end{equation}
where the diffusion term models the uncertainty introduced by loads other than the air conditioners or distributed generations.
This is a simplified model of the utility's revenue or saving: 
we assume that day-ahead purchases are always less than the minimum feasible real-time demand. This is consistent with existing market structures, e.g., the California ISO's.
Therefore the utility company must always purchase some amount of electricity in real time. 
A more detailed revenue model in the real-time market will be studied in the future.

If the utility company offers a contract without taking into account the customers' indoor temperatures and the aggregation of customers accepts it, the customers may experience uncomfortable indoor temperatures during the contract period and choose not to enter a contract with the utility company again. 
There are two major ways for the utility company to take into account the customers' indoor temperatures in the contract:
\begin{enumerate}
\item
the utility company can design a contract assuming that the customers' payoff functions also depend on their indoor temperatures;
\item
the utility company can choose a compensation scheme that incentivizes the customers to maintain their indoor temperatures within $[\underline{y}, \overline{y}]$.
\end{enumerate}
Although the first approach seems to be more natural, it requires 
quantifying the customers' payoff as a function of the indoor temperature: if the utility company underestimates or overestimates the customers' value of their indoor temperatures, the recommended control specified in the contract is no longer incentive compatible.
The second approach is feasible because the utility company has a clear incentive to take into account the customers' indoor temperatures to make the contract desirable enough for the customers to enter into the contract every day. It does not even require modeling the payoff of the customer as a function of the indoor temperature.

\begin{figure}[tb] 
\begin{center}
\includegraphics[width =3.9in]{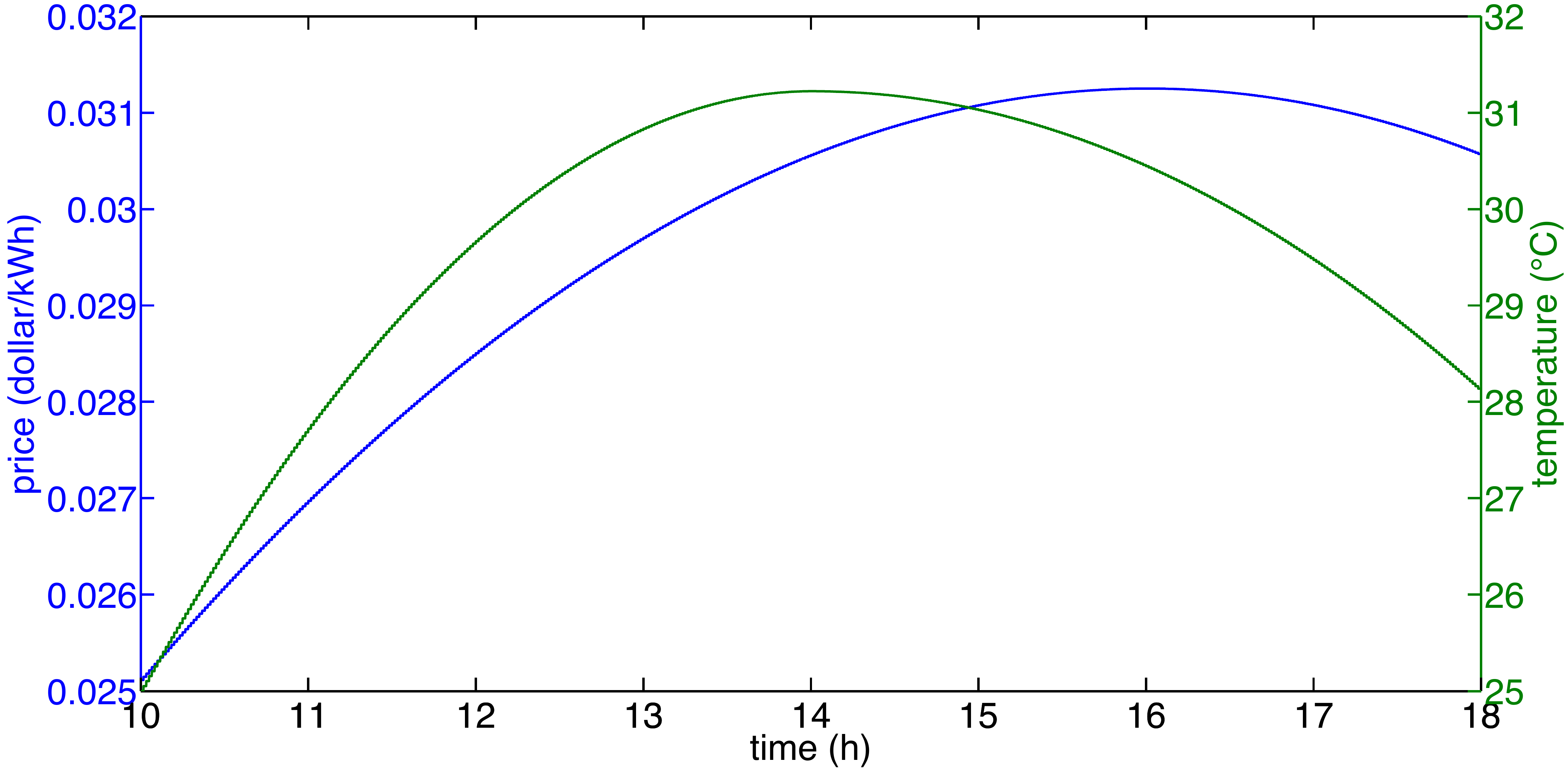}
\caption{Balancing price or LMP, $\lambda_t$, in the real-time market}
 \label{fig:price}
 \end{center}
\end{figure}

In this example, we use the second approach. Let $y_t^i$ be the customer $i$'s indoor temperature at time $t$.
We assume that the outdoor temperature profile $\{\Theta_t\}_{0\leq t \leq T}$ is given.
Then, the customer $i$'s indoor temperature process can be modeled as
\begin{equation}\nonumber
dy_t^i = [\alpha_i(\Theta_t - y_t^i) - \kappa_i u_t^i ]dt,
\end{equation}
for some positive constant $\kappa$ that converts an increase in energy (kWh) to a reduction in temperature ($^\circ$C).
More detailed explanation of this equivalent thermal parameter model can be found in \cite{Sonderegger1978, Yang2014}.
If we assume that  $\alpha_i \equiv \alpha$, $\kappa_i \equiv \kappa$ and $y_0^i \equiv y^0$ for $i=1, \cdots, N$, then the indoor temperatures for all customers have the same dynamics, i.e., $y_t:=y_t^i$, $i=1, \cdots, N$. 

We model the customers' total cost function as the energy cost 
for operating their air conditioners, i.e.,
\begin{equation}\nonumber
\begin{split}
h(u_t) &= \sum_{i=1}^N \zeta u_t^i= \zeta N u_t. 
\end{split}
\end{equation}
Recall that $\zeta$ is the energy price in the customers' tariff.
The utility company cares about the customers' indoor temperature level as well as the compensation that it should transfer to the customers. We therefore model the utility company's running reward as
\begin{equation}\nonumber
r^P(y_t, \pi_t) = - \eta_1 \left[  \exp(\eta_2(y_t - \overline{y})) + \exp(\eta_2(\underline{y} - y_t))  \right ] - \pi_t
\end{equation}
for some positive constants $\eta_1$ and $\eta_2$,
assuming the utility company wants to maintain the indoor temperature within the  range, $[\underline{y}, \overline{y}]$, chosen by the customers when writing the contract. 

\begin{figure}[tb] 
\begin{center}
\includegraphics[width =3.9in]{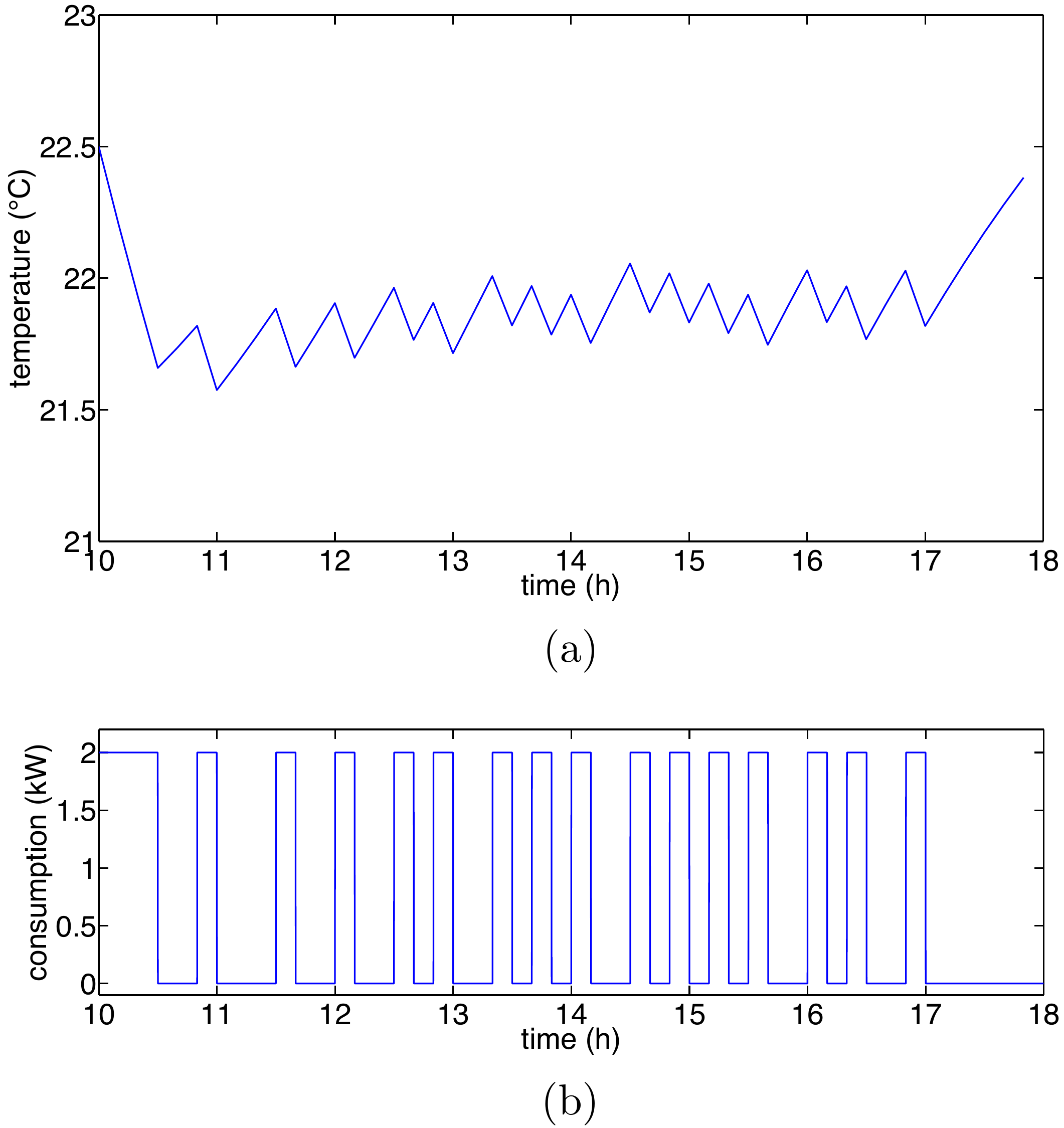}
\caption{(a) Indirectly controlled indoor temperature, and (b) incentive compatible control of each customer.  }
 \label{fig:contract}
 \end{center}
\end{figure}

Once the aggregation of customers receives the real-time compensation, $\pi_t$, from the utility company, we assume that the customers agree to share it equally. That is, the compensation that customer $i$ receives is 
$\pi_t^i = \frac{1}{N} \pi_t =: \bar{\pi}_t$.
The total payoff of the aggregation of customers for the real-time compensation from the utility company can be modeled as
$r^A(\pi_t) =  N  \bar{\pi}_t = \pi_t$.
Similarly, we set $g(C) := C$.

We set the contract period as the period from $10$am to $6$pm and choose $N=1000$, i.e., there are $1000$ customers in the aggregation. 
The total participation payoff is selected as $b = -0.1N$, i.e., each individual's expected payoff is $-\$ 0.1$. 
The real-time balancing price that we use is given in Fig. \ref{fig:price}.
We use $\overline{y} = 22.5$, $\underline{y} = 18$, $y^0 = 22.5$, $\sigma = 200$, $\zeta = 0.2$, $\eta_1 = 10$, $\eta_2 = 5$ and $\kappa = 1$ in the simulation. 
We also set $\mathcal{U} := \{0, 2 \}$ for ON/OFF control of air conditioners and set $\mathcal{P} := [0, 0.2]$.


The incentive compatible control obtained using the proposed method 
is shown in Fig. \ref{fig:contract} (b). 
We notice that each customer turns on his or her air conditioner to 
do pre-cooling  in the early morning.
Fig. \ref{fig:price} shows that the balancing price in the real-time market is low in the morning but has a peak at about $4$pm. Therefore, it is beneficial for the utility company if each customer uses his or her air conditioner in the morning when the balancing price is low and does not use it much when the price is high. Because the compensation induces the desired behavior of the customer, we can see that the dynamic contract is appropriately designed.
Fig. \ref{fig:contract} (a) shows the indoor temperature controlled by the customer. As desired, the indoor temperature stays within the temperature range $[\underline{y}, \overline{y}]$.
More specifically, it decreases a little in the morning (pre-cooling) but keeps slightly increasing after $11$am due to the incentive compatible control shown in Fig. \ref{fig:contract} (a).

\section{Conclusions and Future Work} 

We proposed a dynamic contract design method when the agent has the authority to control both the principal's revenue and an engineered system while the principal has no capability of monitoring the agent's control of the state of the engineered system. 
We showed that Sannikov's idea of using the agent's future expected payoff as a state variable is useful even in our (more general) setting in which the principal has partial observations if the end-time compensation is chosen appropriately and the engineered system is tracked as another index of the agent's performance.
By reformulating the contract design problem as a stochastic optimal control problem, we numerically solved the problem using the associated Hamilton-Jacobi-Bellman equation.
The performance and usefulness of the proposed method were demonstrated by an indirect load control problem.

In the future, we plan to propose a \emph{resilient} contract in a more general setting. 
Furthermore, the application of the proposed method to indirect load control problems should be considered under more realistic assumptions by introducing the heterogeneity of customers and using detailed real-time market with price dynamics.

\section*{Acknowledgements}
The authors would like to thank Professor Lawrence Craig Evans for helpful discussions on the principal-agent problem and PDE-based approaches for it.
This work was supported by the NSF CPS project ActionWebs under grant number 0931843, NSF CPS project FORCES under grant number 1239166,
Robert Bosch LLC through its Bosch Energy Research Network funding program and NSF CPS Award number 1239467.

\bibliographystyle{iEEEtran}

\bibliography{contract_po}
\vfill\eject

\end{document}